\documentclass[12pt,reqno]{amsart}
\usepackage{amsmath,amssymb,latexsym,textcomp,mathrsfs}
\usepackage[all]{xy}
\usepackage{graphicx}
\usepackage{bm,amsmath, amsthm, amssymb, amsfonts}
\usepackage{times}
\usepackage[english]{babel}

\setbox0=\hbox{$+$}
\newdimen\plusheight
\plusheight=\ht0
\def\+{\;\lower\plusheight\hbox{$+$}\;}

\setbox0=\hbox{$-$}
\newdimen\minusheight
\minusheight=\ht0
\def\-{\;\lower\minusheight\hbox{$-$}\;}

\setbox0=\hbox{$\cdots$}
\newdimen\cdotsheight
\cdotsheight=\plusheight%\ht0
\def\cds{\lower\cdotsheight\hbox{$\cdots$}}

\numberwithin{equation}{section}
\theoremstyle{plain}
\newtheorem{theorem}{Theorem}[section]

\newtheorem{corollary}{Corollary}[section]
\newtheorem{definition}{Definition}[section]

\newtheorem{remark}{Remark}[section]
\newtheorem{example}{Example}[section]

  \newenvironment{nouppercase}{%
   \renewcommand{\uppercasenonmath}[1]{}}{}
	
	 \newcommand{\Keywords}[1]{\par\noindent
   {\small{Keywords and phrases}: #1}}
   
   \newcommand{\AMS}[1]{\par\noindent
   {\small{AMS Subject Classification (2010)}: #1}}

\begin{document}

\title {$ (j-i)sg_\kappa^*$-CLOSED  SETS AND PAIRWISE SEMI $ T_\omega$-AXIOM IN BISPACES}
 % \author{Amar Kumar Banerjee$^1$}
 %\author{Jagannath Pal$^2$}
 %\address{1. Department of Mathematics,The University of Burdwan,Golapbag,Burdwan-713104,West Bengal,India}
 %\email{akbanerjee1971@gmail.com}
 %\address{2. Department of Mathematics,The University of Burdwan,Golapbag,Burdwan-713104,West Bengal,India}
 %\email{jpalbu1950@gmail.com}

 \author{Jagannath Pal}
 \author{Amar Kumar Banerjee}
 \newcommand{\acr}{\newline\indent}
 \maketitle
 \address{ Department of Mathematics, The University of Burdwan, Golapbag, East Burdwan-713104,
 West Bengal, India.
 Email: jpalbu1950@gmail.com;   
 Email: akbanerjee@math.buruniv.ac.in.\\}
   
\begin{abstract}
Here we have introduced the ideas of $ (j-i)sg_\kappa^*$-closed sets and  a semi generalized closed set in a bispace; $ i,j=1,2; i\not=j $  and then have studied on pairwise semi $T_0 $-axiom,  pairwise semi $T_1 $-axiom and pairwise semi $T_\omega $-axiom. We have  investigated some of their topological properties and also established a relation among these axioms under some additional conditions. 
\end{abstract}

\begin{nouppercase}
\maketitle
\end{nouppercase}

\let\thefootnote\relax\footnotetext{
\AMS{Primary 54A05, 54D10}

}

\section{\bf Introduction}
\label{sec:int}
For many years several generalization on topological spaces  are done in a more general structure of spaces and obtained interesting results  \cite{CA,WD,DM, MBD}. A. D. Alexandroff  \cite{AD} generalized a topological space  to a $ \sigma $-space (or simply a space) by weakening union requirements where only countable union of open sets were taken to be open. J. C. Kelly \cite{JK} turned his attention to generalize topological spaces to a bitopological space. Bitopological spae was further generalized to a bispace by Lahiri and Das \cite{LK}. Then lot of contribution were made in this direction by several authors \cite{AS, BS, BP, TA}. Simultaneously, closed sets in topological spaces are also generalized in many ways and several topological properties were studied on different spaces. N. Levine \cite{NL} gave the idea of generalized closed sets in a topological space. In 1987, Bhattacharyya and Lahiri \cite{BL} introduced the class of semi-generalized closed sets in  topological spaces. P. Das and Rashid  \cite{DR,RD} defined $ g^* $-closed set, $ sg^* $-closed set in $ \sigma $-spaces and studied their various properties. M. S. Sarsak (2011) \cite{MS} studied $ g_\mu $-closed sets on a generalized topological space. 

In this paper we have studied the notion of semi-generalized closed sets in a bispace $ (X,\kappa_1, \kappa_2) $ by introducing $ (j-i)sg_\kappa^*$-closed sets; $ i,j=1,2; i\not=j $. We have also introduced pairwise semi $T_0 $-axiom,  pairwise semi $T_1 $-axiom, pairwise semi $T_\omega $-axiom in a bispace and  investigated some of their various topological properties and established a relation among these axioms under certain conditions.

 \section{\bf Preliminaries}
 \label{sec:pre}

\begin{definition}\label{1}
\cite{AD}. Let $\mathcal{P}$ be a collection  of subsets of a non empty set $ X $  satisfying the following axioms:

(1)	 The intersection of a countable number of sets in $\mathcal{P}$ is a set in $\mathcal{P}$,     

(2)  The union of a finite number of sets in $\mathcal{P}$ is a set in $\mathcal{P}$,

(3) The empty set $ \emptyset $ and the whole set $ X $ are sets in $\mathcal{P}$,

then the set $ X $ is called an Alexandroff space or a $ \sigma $-space or simply a space  and sets of $\mathcal{P}$ are called closed sets,  complements are open sets. We may take open sets in lieu of closed sets in the definition subject to the conditions that countable summability, finite intersectability, the whole set $ X $ and $ \emptyset $ should be open. Sometimes we denote the collection of all such open sets by $\kappa $ and the $ \sigma $-space by $(X,\kappa)$.  
\end{definition}

\begin{definition}\label{2} (cf. \cite{JK}) . Suppose $ \kappa_1 $ and $ \kappa_2 $ are two collections of subsets of a nonempty set $ X $ such that $ (X, \kappa_1) $ and $ (X, \kappa_2) $ are two $ \sigma $-spaces, then $ (X,\kappa_1, \kappa_2) $ is called a bispace. 
\end{definition}

\begin{definition}\label{3} (cf. \cite{LD}).  A set $ B $ in a bispace $ (X,\kappa_1, \kappa_2) $ is said to be semi $ \kappa_i $-open ($ s\kappa_i $-open in short) if there exists a $ \kappa_i $-open set $ G $ in $ X $ such that $ G\subset B \subset \overline{G_{\kappa_i}} $; $ B $ is said to be semi $ \kappa_i $-closed ($ s\kappa_i $-closed in short) if $ X-B $ is semi $ \kappa_i $-open; $ i=1,2 $ where $ \overline{G_{\kappa_i}} $ is the $ \kappa_i $-closure of $ G $ defined parallely as in a topological space.
\end{definition}

Note that a topological space is a space but in general $ \kappa $ is not a topology. We take the definitions of semi-closure,   semi-interior,  semi-limit point of a set in a $ \sigma $-space similar as in the case of a topological space. Closure (resp. s-closure) of a set may not be  closed (resp. semi closed). The space  $(X, \kappa)$ will simply be denoted by $ X $ and sets  always mean subsets of $ X $.  The complement of a set $ A $ is denoted by  $ A ^c $. Throughout the paper $ R $ and $ Q $ stand respectively for the set of real numbers and the set of rational numbers.

We denote the class of all $ s\kappa_i $-open sets and $ s\kappa_i $-closed sets in a bispace $ (X,\kappa_1, \kappa_2) $ respectively by $ \kappa_i$-$s.o.(X) $ and $\kappa_i$-$s.c.(X);  i=1,2 $.

\begin{definition}\label{4} (cf.\cite{JK}). A  bispace $ (X,\kappa_1, \kappa_2) $ is called pairwise semi-$ T_0 $ if for any pair of distinct points of $X$, either there is a semi $ \kappa_1 $-open set $ U $ such that $ x \in U,  y \not \in U $ or, there is a semi $ \kappa_2 $-open set $ V $ such that $ y\in V , x\not\in V $.
\end{definition}

\begin{definition}\label{5}(cf.\cite{RL}). A bispace $ (X,\kappa_1, \kappa_2) $ is said to be pairwise semi-$ T_1 $ if for any pair of distinct points $ x, y \in X $ there exist semi $ \kappa_1 $-open set  $ U $ and semi $ \kappa_2 $-open set $ V $ such that $ x \in U,  y \not \in U,  y\in V , x\not\in V $.\end{definition}

If $ B$ is a subset of a bispace $ (X,\kappa_1, \kappa_2) $, then we denote semi $ \kappa_i $-closure of $ B $ and semi $ \kappa_i $-derived set of $ B $ respectively by $ \overline{sB_{\kappa_i}} $ and $ s B_{\kappa_i}' $.

\begin{theorem}\label{6}  \cite {PR}.  If $ B\subset (X,\kappa_1, \kappa_2) $, a bispace, then $ \overline{sB_{\kappa_i}} = B\cup s B_{\kappa_i}'; i=1, 2 $.
\end{theorem}

\begin{definition}\label{7} (cf. \cite{RL}).  A bispace $ (X,\kappa_1, \kappa_2) $ is said to be pairwise semi-$ R_0 $  if for every semi $ \kappa_i $-open set $ U, x\in U $ implies $ \overline{s\{x\}_{\kappa_j}}\subset U; i, j=1, 2; i\not=j $.
\end{definition}

\begin{definition}\label{8}  \cite{BL} . A topological space is called  semi-$ T_\frac{1}{2} $ if every $sg$-closed set $ D $ is semi-closed (where $ D $ is $sg$-closed if $ \overline{s(D)}\subset O, O $ is s-open, $ D\subset O $).
\end{definition}

\begin{definition}\cite{MS}\label{9}.
Let $ D $ be a subset in a bispace $ (X,\kappa_1, \kappa_2) $. We define semi-kernel of $ D $ with respect to $ \kappa_i; i=1,2 $ denoted by $ sker_i(D)=\cap \{U: D\subset U; U $ is semi $ \kappa_i $-open\}. 
\end{definition}

\begin{definition} (c.f.\cite{CG}) \label{9A} Two sets $ E,F $ in a bispace $ (X,\kappa_1, \kappa_2) $ are said to be semi $ \kappa_i $-separated if there are two $ s\kappa_i $-open sets $ G, H $ such that $ E\subset G, F\subset H $ and $ E\cap H=F\cap G=\emptyset; i=1,2 $.
\end{definition}

\begin{definition} (c.f.\cite{CG})
$ (X,\kappa_1, \kappa_2) $ is said to be pairwise semi-door bispace if every subset of it is not $ s\kappa_i $-closed then it is $ s\kappa_j $-open; $ i, j=1, 2; i\not=j $.
\end{definition}

\begin{theorem}(c.f.\cite{RD})\label{6A}.
Let $ (X,\kappa_1, \kappa_2) $ be a bispace, then $ s\kappa_i $-closure of a set is $ s\kappa_i $-closed if and only if arbitrary intersection of $ \kappa_i $-closed sets is $ s\kappa_i $-closed in $ X $; $ i=1,2 $.
\end{theorem}

\section{\bf     $ (j-i)sg_\kappa^*$-closed sets  and pairwise semi-$ T_0 $, pairwise semi-$ T_1 $ axioms in bispaces}

In this section we introduce the ideas of $ (j-i)sg_\kappa^*$-closed sets and $ (j-i)sg_\kappa^*$-open sets in a bispace to investigate some topological properties of the sets and to corelate with some new separation axioms like pairwise semi-$ T_0 $, pairwise semi-$ T_1 $ and pairwise semi-$ R_0 $ axioms.

\begin{definition}\label{10}  A set $ B $ of a bispace $ (X, \kappa_1, \kappa_2) $ is said to be semi $g_{\kappa_i}^*$-closed with respect to $s \kappa_j $ denoted by $ (j-i)sg_\kappa^* $-closed if there is a $ s\kappa_i $-closed set  $F$ containing $ B $ such that $ F \subset O$ whenever $B\subset O$ and  $O$ is $ s\kappa_j $-open. $B$ is called semi $ g_{\kappa_i}^*$-open with respect to $s \kappa_j $ denoted by $ (j-i)sg_\kappa^* $-open if $ X - B $ is $ (j-i)sg_\kappa^*$-closed; $ i,j=1,2;i\not=j  $. 
\end{definition}

\begin{theorem}\label{11} A set $ B $ in a bispace $ (X, \kappa_1, \kappa_2) $ is $ (j-i)sg_\kappa^* $-closed if and only if there is a $s \kappa_i $-closed set $ F $ containing $ B $ such that $ F \subset  sker_j(B); i, j=1,2; i\not=j$. 
\end{theorem}
\begin{proof} Suppose $ B\subset X $ and $ B $  is  $ (j-i)sg_\kappa^*$-closed. Then there is a $ s\kappa_i $-closed set $ F $ containing $ B $ such that $ F \subset O$ whenever $B\subset O$ and  $O$ is $ s\kappa_j $-open. Therefore $ F \subset   sker_j(B); i, j=1,2; i\not=j $.

Converse part is obvious.
\end{proof}

\begin{theorem}\label{12}
A set $B$ in a bispace $ (X, \kappa_1, \kappa_2) $ is $ (j-i)sg_\kappa^* $-open if and only if there is a  $s {\kappa_i}$-open set $ V $  contained in $ B $ such that $ F\subset V $ whenever $ F $ is $ s\kappa_j $-closed and $ F\subset B $; $ i,j=1,2; i\not=j $.
\end{theorem}
\begin{proof}
Let $ B $ be a $ (j-i)sg_\kappa^* $-open set; $ i,j=1,2; i\not=j $. Then by Definition \ref{10}, $ X-B $ is a $ (j-i)sg_\kappa^* $-closed set and hence there is a  $ s\kappa_i $-closed set $ F $ containing  $ B^ c $ such that $ F\subset U $ whenever $ B^c\subset U $ and $ U $ is $ s\kappa_j $-open. Therefore, $ F^c\subset B $ such that $ F^c \supset U^c $ whenever $ B\supset U^c $, where $ U^c $ is a $ s\kappa_j $-closed set and $ F^c $ is a $ s\kappa_i $-open set. Hence we can say that $ B $ is $ (j-i)sg_\kappa^* $-open set if and only if there exists a $ s\kappa_i $-open set $ F^c=V, V\subset B $ such that $ F^c=V\supset P=U^c $ whenever $ P\subset B $ and $ P $ is a $ s\kappa_j $-closed.  
\end{proof}

\begin{remark}\label{13}
Every $ s\kappa_i $-closed set $ D $ in a bispace $ (X,\kappa_1, \kappa_2) $ is  $(j-i)sg_\kappa^*$-closed; $ i,j=1,2; i\not=j $ but Example \ref{14}  shows that the converse may not be true; however, it is true if in addition $ D $ is semi $ \kappa_j $-open. Again if a set $ D=sker_j(D) $ then it is $(j-i)sg_\kappa^*$-closed if and only if it is $ s\kappa_i $-closed. 
\end{remark}

\begin{example}\label{14} Suppose $ X=R-Q,  \kappa_1=\{X, \emptyset, \sqrt{\{3\}}\cup G_i\},  \kappa_2=\{X, \emptyset, G_i\}, G_i $ runs over all countable subsets of $ X $.  Then $ (X, \kappa_1, \kappa_2) $ is a bispace but not a bitopological space. Incidentally the bispace is pairwise semi-$ T_1 $ and also pairwise semi-$ T_0 $. Suppose $ A $ is the set of all irrational numbers of $ (1, 2) $. Then $ A $ is not $ s\kappa_1 $-closed as $ X-A $ is not $ s\kappa_1 $-open (for $ \sqrt{3}\in (1, 2) $) but $ A $ is  $ (2-1)sg_\kappa^*$-closed  since $ X $ is the only $ s\kappa_2 $-open set containing $ A $.
\end{example}

\begin{theorem}\label{15}  If a set $B$ of a bispace $ (X, \kappa_1, \kappa_2) $ is $ (j-i)sg_\kappa^* $-closed; $ i,j=1,2; i\not=j$  then there exists a $ s\kappa_i $-closed set $F$ containing $B$ such that $ F - B $ does not contain any non-empty $ s\kappa_j $-closed set; $ i,j=1,2; i\not=j $. 
\end{theorem}

\begin{proof} Let $B$ be a $ (j-i)sg_\kappa^* $-closed set of a bispace $ (X, \kappa_1, \kappa_2) $; $ i,j=1,2; i\not=j$. Then there exists a $ s\kappa_i $-closed set $F$ containing $B$ such that $ F \subset  sker_j(B) $. Let $ P $ be a $ s\kappa_j $-closed set such that $ P\subset F-B $. Then $ X-P $ is $ s\kappa_j $-open and $ B\subset X- P \Rightarrow F\subset X-P \Rightarrow P\subset X-F $. So $ P\subset (X-F)\cap F=\emptyset $. Hence the result follows. 
\end{proof}

But the reverse implication in Theorem \ref{15} is not true as seen from Example \ref{16} although it is true in a space \cite{RD}.

\begin{example}\label{16}
Suppose  $ X = \{a, b, c\}, \kappa_1=\{\emptyset, X, \{a\}, \{a, b\}\}$ and $ \kappa_2=\{\emptyset, X, \{b\}, \{b, c\}\}$.  Then $ (X, \kappa_1, \kappa_2) $ is a bitopological space so a bispace.  Consider the set $ \{b\} $ which is contained in $ s\kappa_1 $-closed set $ \{b,c\} $, so $ \{b,c\}-\{b\}=\{c\}  $ does not contain any non-empty $ s\kappa_2 $-closed set. Since $ \{b\} $ itself is a $ s\kappa_2 $-open set, there is no $ s\kappa_1 $-closed set containing $ \{b\} $ contained in $ \{b\} $; and hence $ \{b\} $ is not $ (2-1)sg_\kappa^* $-closed.  
\end{example}

\begin{theorem}\label{17}
Suppose $ D $ is a $ (i-j)sg_\kappa^* $-closed set in a bispace $ (X, \kappa_1, \kappa_2) $; $ i,j=1,2; i\not=j $, then $ D $ is $ s\kappa_j $-closed if and only if $ \overline{sD_{\kappa_j}} $ and $ \overline{sD_{\kappa_j}}-D $ are $ s\kappa_j $-closed. 
\end{theorem}
\begin{proof}
Let the $ (i-j)sg_\kappa^* $-closed set $ D $ be a $ s\kappa_j $-closed set in a bispace $ (X, \kappa_1, \kappa_2) $ then $ \overline{sD_{\kappa_j}}=D $ is $ s\kappa_j $-closed and  $ \overline{sD_{\kappa_j}}-D=\emptyset $ is $ s\kappa_j $-closed; $ i,j=1,2; i\not=j $.

Conversely, let $ \overline{sD_{\kappa_j}} $ and $ \overline{sD_{\kappa_j}}-D $ be $ s\kappa_j $-closed. Since $ D $ is $ (i-j)sg_\kappa^* $-closed, then there is a $ s\kappa_j $-closed set $ P, P\supset D $ such that $ P-D $ does not contain any non-empty $ s\kappa_i $-closed set. So $ \overline{sD_{\kappa_j}}-D\subset P-D\Rightarrow  \overline{sD_{\kappa_j}}-D=\emptyset\Rightarrow  \overline{sD_{\kappa_j}}=D\Rightarrow D $ is $ s\kappa_j $-closed.
\end{proof}

Intersection and union two $ (j-i)sg_\kappa^* $-closed sets in a bispace $ (X, \kappa_1, \kappa_2) $ are not in general a  $ (j-i)sg_\kappa^* $-closed;  $ i,j=1,2; i\not=j $ as shown by the Examples \ref{18} (i) and \ref{18} (ii). However, Theorems \ref{20}and \ref{22} show that these are true under certain additional conditions.

\begin{example}\label{18} (i) : 
Suppose  $ X = \{a, b, c\},  \kappa_1=\{\emptyset,X,  \{a, b\}\}$ and $ \kappa_2=\{\emptyset,X,\{b\}\}$.  Then $ (Y, \kappa_1, \kappa_2) $ is a bitopological space so a bispace. Now a family of $ s\kappa_1$-open sets is  $\{\emptyset,X,  \{a, b\}\} $ and a family of $  s\kappa_2$-open sets is  $\{\emptyset,X,  \{b\}\} $.   Then clearly the sets $\{a, b\} $ and $ \{b,c\} $ are $ (2-1)sg_\kappa^* $-closed sets but $ \{a, b\} \cap \{b,c\}=\{b\} $ is not $ (2-1)sg_\kappa^* $-closed set.

(ii): 
Suppose $ X=R-Q $ and $ \kappa_1 = \{X, \emptyset, G_i\cup \{\sqrt{3} \}$ where $ G_i $s are countable subsets of $ X $ and $ \kappa_2 = \{X, \emptyset, G_i \}$ where $ G_i $s are countable subsets of $ X- \{\sqrt{2}\} $. Then $ (X, \kappa_1,\kappa_2) $ is a bispace, but not a bitopological space. Suppose $ A = X - \{\sqrt{2}, \sqrt{3}, \sqrt{5}\} $ and $ B = X - \{\sqrt{2}, \sqrt{7}, \sqrt{11}\} $. Then $ X- A $ is $ s\kappa_2 $-open since $ \{\sqrt{3}, \sqrt{5}\} $ is a $ \kappa_2 $-open set and $ \{\sqrt{3}, \sqrt{5}\}  \subset X - A \subset \overline{\{\sqrt{3}, \sqrt{5}\}} $. Similarly, $ X - B $ is also $ s\kappa_2 $-open. Therefore $ A $ and $ B $ are $ s\kappa_2 $-closed sets and hence $ A $ and $ B $ are $ (1-2)sg_\kappa^* $-closed sets.  Again take $ C = A \cup B = (X - \{\sqrt{2}, \sqrt{3}, \sqrt{5}\})\cup (X - \{\sqrt{2}, \sqrt{7}, \sqrt{11}\}) = X - \{ \sqrt{2}\} $. Since $ \{\sqrt{2}\}$ is not  $ s\kappa_2 $-open, then $ X - \{\sqrt{2}\} $ is not $ s\kappa_2 $-closed, so $ X $ is the only $ s\kappa_2 $-closed set containing $ C $. But $ C $ is a $ s\kappa_1 $-open set. Hence by Definition \ref{10}, $ C $ is not  a $ (1-2)sg_\kappa^* $-closed  set.  
\end{example}

\begin{theorem} \label{20}  Union of two $ (j-i)sg_\kappa^* $-closed sets in a bispace $(X,\kappa_1, \kappa_2) $ is $ (j-i)sg_\kappa^* $-closed if union of two $ s\kappa_i $-closed sets is $ (j-i)sg_\kappa^* $-closed set;  $ i,j=1,2; i\not=j $. \end{theorem}
\begin{proof}
Let $ E, F $ be two $ (j-i)sg_\kappa^* $-closed sets in a bispace $ (X,\kappa_1, \kappa_2);  i,j=1,2;i\not=j $. Assume $ E\cup F \subset G, G $ is $ s\kappa_j $-open. So $ E\subset G, F\subset G $ and there are $ s\kappa_i $-closed sets $ P,Q $ containing respectively $ E, F $ such that $ P\subset G, Q\subset G $. Hence  $ P\cup Q\supset E\cup F$ and $ P\cup Q \subset G $.  By our assumption $ P\cup Q $ is $ (j-i)sg_\kappa^* $-closed and so there is a $ s\kappa_i $-closed set $ K $ containing $ P\cup Q $ such that $ K\subset G\Rightarrow E\cup F\subset P\cup Q\subset K\subset G $. Hence the result follows.
\end{proof}

\begin{theorem} \label{22}    Union of two semi $ \kappa_j $-separated $ (j-i)sg_\kappa^* $-open sets in a bispace $(X,\kappa_1, \kappa_2) $ is $ (j-i)sg_\kappa^* $-open;  $ i,j=1,2; i\not=j  $.\end{theorem}
\begin{proof}
Suppose $ E_1,E_2 $ are two semi $ \kappa_j $-separated $ (j-i)sg_\kappa^* $-open sets; $ i,j=1,2;i\not=j $. Then for semi $ \kappa_j $-separated sets there are $ s\kappa_j $-open sets $ G_1, G_2 $ such that $ E_1\subset G_1, E_2\subset G_2 $ and $ E_1\cap G_2=E_2\cap G_1=\emptyset $. Let $ D_1=G_1^c, D_2=G_2^c $. Then $ D_1, D_2 $ are $ s\kappa_j $-closed sets and $ E_1\subset D_2, E_2\subset D_1 $. Again if $ P_n\subset E_n, n=1,2; P_n $ is  $ s\kappa_j $-closed set then for $ (j-i)sg_\kappa^* $-open sets, there is $ s\kappa_i $-open set $ U_n\subset E_n $ such that $ P_n\subset U_n $ by Theorem \ref{12}. Clearly, $ U_1\cup U_2 $ is a $ s\kappa_i $-open set and $ U_1\cup U_2\subset E_1\cup E_2 $.   Assume $ D $ is a $ s\kappa_j $-closed set and $ D\subset E_1\cup E_2 $.  Now $ D= D\cap (E_1\cup E_2 )=(D\cap E_1)\cup (D\cap E_2)\subset (D\cap D_2)\cup (D\cap D_1) $ where $(D\cap D_2)$ and $ (D\cap D_1)$ are  $ s\kappa_j $-closed sets. Further $ (D\cap D_1)\subset (E_1\cup E_2)\cap D_1= (D_1\cap E_1)\cup (D_1\cap E_2)=\emptyset\cup (D_1\cap E_2)\subset E_2$ and so $ (D\cap D_1)\subset U_2 $. Similarly, $ (D\cap D_2)\subset U_1 $. Then $ D\subset  U_1\cup U_2 $ and hence the result follows. 
\end{proof}

\begin{theorem}\label{61}
Let $ (X,\kappa_1, \kappa_2) $ be a bispace. If $ C $ is $ (j-i)sg_\kappa^*$-closed and $ C\subset D\subset \overline{sC_{\kappa_i}} $ then $ D $ is $ (j-i)sg_\kappa^*$-closed; $ i,j=1,2; i\not=j $.
\end{theorem}

\begin{proof}
Suppose $ C $ is a $ (j-i)sg_\kappa^*$-closed set in $ (X,\kappa_1, \kappa_2),   C\subset D\subset \overline{sC_{\kappa_i}} $ and $ D\subset U, U $ is $ s\kappa_j $-open; $ i,j=1,2; i\not=j $. Then $ C\subset U $ and hence  there is a $ s\kappa_i $-closed set $ P $ containing the set set $ C $ such that $ P\subset U $. Now $ C\subset P\Rightarrow \overline{sC_{\kappa_i}}\subset P \Rightarrow P\supset \overline{sC_{\kappa_i}}\supset D \Rightarrow  D $ is  $ (j-i)sg_\kappa^*$-closed.
\end{proof}

\begin{theorem}\label{62}
Let $ (X, \kappa_1, \kappa_2) $ be a bispace and $ C\subset B $ where $ B $ is $ s\kappa_j $-open and $ (j-i)sg_\kappa^* $-closed; $ i,j=1,2; i\not=j $. Then $ C $ is $ (j-i)sg_\kappa^* $-closed if and only if $ C $ is $ (j-i)sg_\kappa^* $-closed relative to $ B $.
\end{theorem}

\begin{proof}
Let $ (X, \kappa_1, \kappa_2) $ be a bispace and $ C\subset B $ where $ B $ is $ s\kappa_j $-open and $ (j-i)sg_\kappa^* $-closed; $ i,j=1,2; i\not=j $. Then by Remark \ref{13}, $ B $ is $ s\kappa_i $-closed. 
Let $ C $ be $ (j-i)sg_\kappa^*$-closed. Since $ B $ is $ s\kappa_j $-open and $ C\subset B $, there is a $ s\kappa_i $-closed set $ P_1 $ containing the set $ C $ such that $ P_1\subset B $ and so $ P_1 $ is  $ s\kappa_i $-closed in $ B $. Now let $ C\subset G $ where $ G $ is $ s\kappa_j $-open in $ B $.  Evidently $ G $ is $ s\kappa_j $-open in $ X $ and hence $ P_1\subset G\Rightarrow C $ is $ (j-i)sg_\kappa^*$-closed in $ B $.

Conversely, let $ C $ be $ (j-i)sg_\kappa^* $-closed relative to $ B $. Then there is a $ s\kappa_i $-closed set $ P_2 $ in $ B $ containing the set $ C $ such that $ P_2\subset G' $ where $ G' $ is $ s\kappa_j $-open in $ B $ containing the set $ C $. As $ B $ is $ s\kappa_i $-closed, $ P_2 $ is  $ s\kappa_i $-closed in $ X $. Let $ C\subset H,H $ is $ s\kappa_j $-open, then $ C\subset H\cap B $, a $ s\kappa_j $-open set in $ B $. Hence $ C\subset P_2\subset H\cap B\Rightarrow C\subset P_2\subset H\Rightarrow C $ is $ (j-i)sg_\kappa^* $-closed.
\end{proof}

\begin{corollary}\label{63}
Let $ B $ be $ s\kappa_j $-open and $ (j-i)sg_\kappa^* $-closed set in a bispace $ (X,\kappa_1,\kappa_2) $, then $ B\cap C $ is $ (j-i)sg_\kappa^* $-closed  if $ C $ is $ (j-i)sg_\kappa^* $-closed; $ i,j=1,2; i\not=j $.
\end{corollary}

\begin{proof}
Obviously $ B $ is $ s\kappa_i $-closed, so $ B\cap C $ is $ s\kappa_i $-closed in $ C $ which implies that  $ B\cap C $ is $ (j-i)sg_\kappa^* $-closed  in $ C $. Then by Theorem \ref{62}, $ B\cap C $ is $ (j-i)sg_\kappa^* $-closed.
\end{proof}

\begin{theorem}\label{64}
If each subset of a bispace $ (X,\kappa_1,\kappa_2) $ is $ (j-i)sg_\kappa^* $-closed, then $ \kappa_i$-$s.c.(X)=\kappa_j$-$s.o.(X); i\not=j; i,j=1,2 $.
\end{theorem}

\begin{proof}
Let each subset be $ (j-i)sg_\kappa^* $-closed and $ B\in \kappa_j$-$s.o.(X) $; $ i,j=1,2; i\not=j $. Then  $ B $ is $ (j-i)sg_\kappa^* $-closed, so there is a $ s\kappa_i $-closed set $ P $ containing the set $ B $ such that $ P\subset B $ which implies that $ P=B\Rightarrow B $ is $ s\kappa_i $-closed and so $ \kappa_j$-$s.o.(X)\subset \kappa_i$-$s.c.(X) $. Again let $ E\in \kappa_i$-$s.c.(X)\Rightarrow X-E\in \kappa_i$-$s.o.(X) $, then  by same argument, $ X-E $ is $ s\kappa_j $-closed and hence $ E $ is $ s\kappa_j $-open. So $ \kappa_i$-$s.c.(X)\subset \kappa_j$-$s.o.(X)\Rightarrow \kappa_i$-$s.c.(X)= \kappa_j$-$s.o.(X) $.
\end{proof}

The reverse implication of  above Theorem is not true as revealed from the Example \ref{65}. But the converse is true under some additional condition shown in the next Theorem \ref{66}.

\begin{example}\label{65} Suppose $ X=R-Q, \kappa_1=\kappa_2=\{\emptyset, X, G_n, D_n\} $, where $ G_n $ and $ D_n $ run over respectively all countable sets and cocountable sets of $ X $. Then  $ (X,\kappa_1,\kappa_2) $ is a bispace but not a bitopological space. Clearly, we have $ \kappa_i$-$s.c.(X)= \kappa_j$-$s.o.(X); i\not=j; i,j=1,2  $. Now consider the set $ D $ of all irrational numbers in $ (0, 2) $. Then $ D\not\in \kappa_i$-$s.c.(X)\Rightarrow D $ is not $ (j-i)sg_\kappa^* $-closed, by Remark \ref{13}, since $ D=sker_j(D) $.
\end{example}

\begin{theorem}\label{66}
Let $ (X,\kappa_1,\kappa_2) $ be a bispace with $ \kappa_i$-$s.c.(X)=\kappa_j$-$s.o.(X); i\not=j; i,j=1,2 $, then each subset is $ (j-i)sg_\kappa^* $-closed if and only if the bispace satisfies the condition (C). (C): Arbitrary intersection of $ \kappa_i $-closed sets is $ s\kappa_i $-closed in $ X $.
\end{theorem}

\begin{proof}
Let $ (X,\kappa_1,\kappa_2) $ be a bispace with $ \kappa_i$-$s.c.(X)=\kappa_j$-$s.o.(X); i\not=j; i,j=1,2 $. 

Suppose that every subset is $ (j-i)sg_\kappa^* $-closed and $ P_n $ is an arbitrary collection of $ \kappa_i $-closed sets and $ P=\bigcap P_n $. Then $ P $ is $ (j-i)sg_\kappa^* $-closed. So there is a $ s\kappa_i $-closed set $ P'$ containing $ P $ such that $ P'\subset V $ whenever $ P\subset V, V $ is $ s\kappa_j $-open. Now since each $ P_n\in \kappa_i$-$s.c.(X)=\kappa_j$-$s.o.(X) $ and $ P\subset P_n $, it follows that $ P'\subset P_n $, for all $ n $ i.e. $ P'\subset \bigcap P_n=P $. Hence $ P'=P $ and the result follows. 

Conversely, let the condition (C) hold and $ D\subset X $. So $ \overline{(D)_{\kappa_i}} $ is $ s\kappa_i $-closed and hence by Theorem \ref{6A},  $ \overline{s(D)_{\kappa_i}}=P $ (say) is $ s\kappa_i $-closed. Let $ D\subset U,U $ is $ s\kappa_j $-open. Since $ \kappa_i$-$s.c.(X)=\kappa_j$-$s.o.(X); i\not=j; i,j=1,2; U $ is $ s\kappa_i $-closed. So $ P=\overline{s(D)_{\kappa_i}}\subset \overline{s(U)_{\kappa_i}}=U $. So $ D $ is $ (j-i)sg_\kappa^* $-closed.   
\end{proof}

\begin{theorem}\label{23} Suppose $(X,\kappa_1, \kappa_2) $ is a bispace. For each $ x \in X $, if $\{x\}$ is not $ s\kappa_i $-closed then $ \{x\}^c $ is $ (i-j)sg_\kappa^* $-closed; $ i,j=1,2; i\not=j $.
\end{theorem}
\begin{proof}
Suppose $ x\in X $ and $ \{x\} $ is not $ s\kappa_i $-closed, then $ \{x\}^c $ is not $ \kappa_i $-open. Therefore $ X $ is the only $ s\kappa_i $-open set which contains $ \{x\}^c $. Also $ X $ may be taken as $ s\kappa_j $-closed set containing $ \{x\}^c $. Therefore $ \{x\}^c $ is a $ (i-j)sg_\kappa^* $-closed; $ i,j=1,2; i\not=j $. 
\end{proof}

\begin{theorem}\label{24} 
A bispace $(X,\kappa_1, \kappa_2) $  is pairwise semi-$T_0$ if and only if for any pair of distinct points $x, y \in X $,  there is a set $B$ which contains only one of them such that $ B $ is either $ s\kappa_i $-open or $ s\kappa_j $-closed, $ i, j=1,2 $.
\end{theorem}

\begin{proof} 
Suppose that the bispace $(X,  \kappa_1, \kappa_2)$  is  pairwise semi-$T_0$ and  $x, y \in X, x\not=y $. So there exists a $ s\kappa_i $-open set $ B $ containing one of $ x, y $ say $ x $, but not $ y $. Then it is done. There may arise the case that $ x\in B $ and $ y \not\in B $. But $ B $ is not a $ s\kappa_i $-open set, then $ y\in X-B, x\not\in X-B$ and $ X-B $ is $ s\kappa_j $-open. So $ B $ is  $ s\kappa_j $-closed. 

Conversely, suppose the condition holds and $ B $ is a $ s\kappa_i $-open set containing $ x $ only. Then clearly the bispace is pairwise semi-$ T_0 $. If $ x\in B $ but $ B $ is $ s\kappa_j $-closed, then $ X-B $ is $ s\kappa_j $-open and $ y\in X-B, x\not\in X-B $. Hence in this case also the bispace is pairwise semi-$ T_0 $.
\end{proof}

\begin{remark}\label{25} If $ (X, \kappa_1) $ or $ (X, \kappa_2) $ are semi-$ T_0 $  spaces, then the bispace $(X,  \kappa_1, \kappa_2)$  is pairwise semi-$ T_0 $. Again, it can be easily shown that pairwise semi-$ T_1 $ is pairwise semi-$T_0$ but the converse may not be true as seen from the Example \ref{26} given below.
\end{remark}

\begin{example}\label{26} Example of pairwise semi-$ T_0 $ bispace which is not pairwise semi-$ T_1 $.

Suppose $ X=R-Q, \kappa_1=\{X, \emptyset, \{\sqrt{3}, \sqrt{5}\}\cup G_i\}, \kappa_2=\{X, \emptyset, G_i\} $ where $ G_i $ runs over all countable subsets of $ X $. Then $(X,  \kappa_1, \kappa_2)$ is a bispace but not a bitopological space. Clearly, the bispace is pairwise semi-$ T_0 $ but it is not pairwise semi-$ T_1 $. For, consider the pair of distinct points $ \sqrt{3}, \sqrt{5}\in X $ and it will not be possible to find a $ s\kappa_1 $-open set which contains only one of $ \sqrt{3}, \sqrt{5} $.  So $ (X,  \kappa_1, \kappa_2 )$ is not pairwise semi-$ T_1 $ bispace.
\end{example}

Now we are going to find out a condition for a pairwise semi-$ T_0 $ bispace to be a pairwise semi-$ T_1 $.

\begin{theorem}\label{27} If a bispace $ (X,  \kappa_1, \kappa_2 )$ is pairwise semi-$T_0 $  then for each pair of distinct points $ p,q\in X $,  either $ p\not \in\overline{s\{q\}_{\kappa_1}} $  or $ q\not \in \overline
{s\{p\}_{\kappa_2}} $.\end{theorem}
\begin{proof}
Let $(X,  \kappa_1, \kappa_2)$ be a pairwise semi-$T_0 $ bispace and $ p, q\in X, p \not= q $. Since $ X $ is pairwise  semi-$T_0 $, there exists a $ s\kappa_1 $-open set $ U $ which contains only one of $ p, q $. Suppose that $ p\in U $ and $ q \not \in U $. Then the $ s\kappa_1 $-open set $ U $ has an empty intersection with $ \{q\} $. Hence $ p\not\in s\{q\}_{\kappa_1}'$. Since $ p \not = q,  p\not\in\overline{s\{q\}_{\kappa_1}}$. Similarly, there may exist a $ s\kappa_2 $-open set $ V $ such that $ q \in V, p \not\in V $ and we get $ q\not\in \overline{s\{p\}_{\kappa_2}} $. Hence the result follows.
\end{proof}

\begin{theorem}\label{28}
In a pairwise semi-$ T_0 $ bispace,  $ p,q\in X, p\not=q $ implies $ \overline{s\{p\}_{\kappa_2}}\not=\overline{s\{q\}_{\kappa_1}} $.
\end{theorem}

\begin{proof}
Let the bispace $(X,  \kappa_1, \kappa_2)$ be  pairwise semi-$T_0 $ and $ p, q\in X, p \not= q $. Then either $ p\not \in\overline{s\{q\}_{\kappa_1}} $  or $ q\not \in \overline
{s\{p\}_{\kappa_2}}$. Let $ p\not \in\overline{s\{q\}_{\kappa_1}} $, but $ p \in \overline
{s\{p\}_{\kappa_2}}$. Thus $ \overline{s\{p\}_{\kappa_2}}\not=\overline{s\{q\}_{\kappa_1}}$. 
\end{proof}

\begin{theorem}\label{29}
Pairwise semi-$ T_0 $ and pairwise semi-$ R_0 $ bispace is pairwise semi-$ T_1 $.
\end{theorem}
\begin{proof}
Let the bispace $ (X, \kappa_1, \kappa_2) $ be pairwise semi-$ T_0 $ and pairwise semi-$ R_0 $ and $ p,q\in X; p\not=q $. By Theorem \ref{27}, either $ p\not \in\overline{s\{q\}_{\kappa_1}} $  or $ q\not \in \overline
{s\{p\}_{\kappa_2}} $. Let $ q\not \in \overline
{s\{p\}_{\kappa_2}} $, then there is a $ s\kappa_2 $-closed set $ P $ containing $ p $ such that $ q\not\in P $. So $ q\in X- P $ which is a $ s\kappa_2 $-open set and $ p\not\in X- P $. From Definition \ref{7} of  pairwise semi-$ R_0 $ bispace, $ \overline{s\{q\}_{\kappa_1}}\subset X-P\Rightarrow \overline{s\{q\}_{\kappa_1}}\cap P=\emptyset\Rightarrow \overline{s\{q\}_{\kappa_1}}\cap \{p\}=\emptyset\Rightarrow p\not\in\overline{s\{q\}_{\kappa_1}} $. As $ p\not=q $, then $ p $ is not a $ s\kappa_1 $-limit point of $ \{q\} $. Then there is a $ s\kappa_1 $-open set $ U $ such that $ p\in U $ but $ q\not\in U $. So $ p\in U,q\in X-P,p\not\in X-P, q\not\in U  $. Thus the bispace is pairwise  semi-$ T_1 $.
\end{proof}

The condition pairwise semi $ R_0 $ in the above Theorem is not necessary for pairwise  semi-$ T_1 $ as shown by Example \ref{30} although pairwise  semi-$ T_1 $  implies pairwise  semi-$ T_0 $.

\begin{example}\label{30} Example of pairwise  semi-$ T_1\not\Rightarrow $ pairwise  semi-$ R_0 $.

Let $ X=\{a, b,c\}, \kappa_1=\{\emptyset, X, \{a\}, \{c\}, \{a,c\}\}, \kappa_2=\{\emptyset, X, \{b\},\{a,b\}\} $. Then $ (X, \kappa_1, \kappa_2) $ is a bitopological space so a bispace. Here a family of $ s\kappa_1 $-open  sets is $ \{\emptyset, X, \{a\}, $  $\{c\}, \{a,c\},$  $ \{a,b\}, \{b,c\}\} $ and a family of $ s\kappa_2 $-open sets is $ \{\emptyset, X, \{b\}, \{a,b\}, \{b,c\}\} $ and the bispace is pairwise semi-$ T_1 $. Now consider $ b\in   \{a,b\} $, a $ s\kappa_1 $-open  set, then $ \overline{s\{b\}_{\kappa_2}}=X\not\subset \{a,b\} $ which implies that $ (X, \kappa_1, \kappa_2) $ is not pairwise semi-$ R_0 $.  
\end{example}

\begin{definition}\label{31}   A bispace $ (X, \kappa_1, \kappa_2) $ is said to be pairwise semi-symmetric if for any $ x,  y \in X,  x \in \overline{s\{y\}_{\kappa_i}} \Longrightarrow y\in \overline{s\{x\}_{\kappa_j}}, i,j=1,2; i\not=j $. \end{definition}

\begin{theorem}\label{32} Pairwise semi-symmetric, pairwise semi-$ T_0 $ bispace is pairwise semi $ T_1 $.
\end{theorem}

\begin{proof}
Let $(X,  \kappa_1, \kappa_2)$ be a pairwise semi-symmetric and pairwise semi-$ T_0 $ bispace and let $ a, b \in X,  a\not= b$. Since the bispace is pairwise semi-$ T_0 $, either $ a\not\in\overline{s\{b\}_{\kappa_1}} $ or $ b\not\in\overline{s\{a\}_{\kappa_2}} $. Let $ a\not\in\overline{s\{b\}_{\kappa_1}} $ . Then we claim that $ b\not\in\overline{s\{a\}_{\kappa_2}} $.  For, if $ b\in \overline{s\{a\}_{\kappa_2}} $  then it would imply that $ a\in\overline{s\{b\}_{\kappa_1}} $, since the bispace is pairwise semi-symmetric. But this contradicts that $   a\not\in\overline{s\{b\}_{\kappa_1}} $. Since $ a\not\in \overline{s\{b\}_{\kappa_1}}$, there is a $ s\kappa_1 $-closed set $ F $ such that $ b\in F $ and $ a\not\in F $. So $ a\in X-F $, a $ s\kappa_1 $-open set and $ b\not\in X-F $. Again since $ b\not\in\overline{s\{a\}_{\kappa_2}} $, there is a $ s\kappa_2 $-closed set $ P $ such that $ a\in P $ and $ b\not\in P $.  So $ b\in X-P $, a $ s\kappa_2 $-open set and $ a\not\in X-P $.  Hence the bispace is pairwise semi-$ T_1 $.  
\end{proof}

\section{\bf Pairwise semi-$ T_\omega $ bispace}

In this section we give the idea of pairwise semi-$ T_\omega $ axiom in a bispace and discuss some of its properties. Though pairwise semi-$ T_\omega $ axiom can not be placed between pairwise semi-$ T_0 $ and pairwise semi-$ T_1 $ axioms but we try to establish a relation among them.

\begin{definition} \label{33}  A bispace $ (X, \kappa_1, \kappa_2) $ is said to be pairwise semi-$T_\omega $  if and only if every $ (j-i)sg_\kappa^* $-closed set is $ s\kappa_i $-closed; $ i,j=1,2; i\not=j $.
\end{definition}

A topological space is $ T_\frac{1}{2} $ \cite{WD} if and only if each singleton is either open or closed. For a $ \sigma $-space, condition is necessary but it is not sufficient \cite{DR}. For this we introduce the following definition to find 
necessary and sufficient conditions for a bispace to be pairwise semi-$ T_\omega $.

\begin{definition} \label{34}   For any set $ D $ in a bispace $ (X, \kappa_1, \kappa_2) $, we define $ \overline{sD_{gi}^{*j}}=\bigcap\{A:D\subset A, A $ is $ (j-i)sg_\kappa^* $-closed in $ X $\}, then $ \overline{sD_{gi}^{*j}} $ is called $ (j-i)sg_\kappa^* $-closure of $ D; i,j=1,2; i\not=j$.

We denote the following sets as $ \mathcal {G}_i $  and $ \mathcal {G}_i' $ which will be used in the sequel.

(i) $ \mathcal {G}_i = \{A: \overline{s(X -A)_{\kappa_i}}$ is $ s\kappa_i $-closed\}; $ i=1,2 $  and

(ii) $ \mathcal {G}_i' = \{A: \overline{s(X -A)_{gi}^{*j}}$ is  $ (j-i)sg_\kappa^*$-closed; $ i,j=1, 2;i\not =j $. \end{definition}

\begin{theorem}\label{60}
Suppose $ (X, \kappa_1, \kappa_2) $ is a pairwise semi-$ T_\omega $ bispace then 

(i): for each $ x\in X, $ if $ \{x\} $ is is not $ s\kappa_i $-closed, it is $ s\kappa_j $-open; $ i,j=1,2; i\not=j $,

(ii): $ \kappa_i$-$s.o.(X)=\kappa_{i-j}$-$s.o.(X)^* $  where $ \kappa_{i-j}$-$s.o.(X)^*=\{D: D\subset X: \overline{s(X-D)_{gi}^{*j}}=X-D; i\not=j; i,j=1,2 $.  
\end{theorem}

\begin{proof}
Let $ (X,  \kappa_1, \kappa_2) $ be a pairwise semi-$ T_\omega $ bispace.

(i):  Let $ x\in X, \{x\} $ is not $ s\kappa_i $-closed, then by Theorem \ref{23}, $ \{x\}^c $ is  $ (i-j)sg_\kappa^*$-closed, so $ \{x\}^c $ is $s\kappa_j$-closed, $ i,j=1,2;i\not=j $. This implies that $ \{x\} $ is $ s\kappa_j$-open.

(ii): For pairwise semi-$ T_\omega $ bispace, $ s\kappa_i $-closed sets and $ (j-i)sg_\kappa^* $-closed sets coincide and hence for each $ D\subset X $ we have $ \overline{s
(D)_{\kappa_i}}=\overline{sD_{gi}^{*j}} $; $ i,j=1,2; i\not=j $. Hence the result follows.
\end{proof}

\begin{theorem} \label{35}   A bispace $ (X,  \kappa_1, \kappa_2) $ is pairwise semi-$ T_\omega $ if and only if 

(a)  for each $ x \in X $,  $ \{x\} $ is not $ s\kappa_i $-closed, then it is $ s\kappa_j $-open; $ i,j=1,2;i\not=j $ and

(b)  $ \mathcal {G}_i =  \mathcal {G}_i' $  
where $ \mathcal {G}_i $ and $  \mathcal {G}_i' $ are as in Definition \ref{34} 
\end{theorem}
\begin{proof}
Suppose $ (X,  \kappa_1, \kappa_2) $ is a pairwise semi-$ T_\omega $ bispace and $ x\in X, \{x\} $ is not $ s\kappa_i $-closed, then by Theorem \ref{23}, $ \{x\}^c $ is  $ (i-j)sg_\kappa^*$-closed and so $ \{x\}^c $ is $s\kappa_j$-closed, $ i,j=1,2;i\not=j $. This implies that $ \{x\} $ is $ s\kappa_j$-open. Now let $ D\in \mathcal{G}_i $, then $ \overline{s(X-D)_{\kappa_i}} $ is $ s\kappa_i $-closed  $ \Rightarrow \overline{s(X-D)_{\kappa_i}} $ is $ (j-i)sg_\kappa^*$-closed $ \Rightarrow \overline{s(X-D)_{gi}^{*j}} $ is $ (j-i)sg_\kappa^*$-closed $ \Rightarrow D\in \mathcal{G}_i' \Rightarrow \mathcal {G}_i \subset  \mathcal {G}_i' $. Again suppose $ D\in \mathcal{G}_i' $, then $ \overline{s(X-D)_{gi}^{*j}} $ is $ (j-i)sg_\kappa^*$-closed $ \Rightarrow \overline{s(X-D)_{\kappa_i}} $ is $s\kappa_i $-closed $ \Rightarrow D\in \mathcal{G}_i \Rightarrow \mathcal {G}_i' \subset  \mathcal {G}_i $. Hence $ \mathcal {G}_i =  \mathcal {G}_i' $.

Conversely, suppose the conditions hold and $ D $ is $ (j-i)sg_\kappa^*$-closed. Then $ \overline{s(D)_{g_i}^{*j}}=D $ is $ (j-i)sg_\kappa^*$-closed $ \Rightarrow X-D \in \mathcal {G}_i' $. Since $ \mathcal {G}_i =  \mathcal {G}_i' $, then $ X-D \in \mathcal {G}_i \Rightarrow \overline{s(D)_{\kappa_i}}$ is  $ s\kappa_i$-closed. We claim that $ \overline{s(D)_{\kappa_i}} =D$. If not, then there exists $ x\in  \overline{s(D)_{\kappa_i}}-D $. Since $ D $ is $ (j-i)sg_\kappa^*$-closed, $\{x\} $ cannot be $ s\kappa_j $-closed, by Theorem \ref{15}. Hence by assumption, $ \{x\} $ is $s\kappa_i $-open and so $ \{x\}^c $ is  $s\kappa_i$-closed. But $ x\not \in D \Rightarrow D\subset \{x\}^c \Rightarrow \overline{s(D)_{\kappa_i}} \subset \{x\}^c$. We see that $ x\in \overline{s(D)_{\kappa_i}} -D $ and $ \overline{s(D)_{\kappa_i}} \subset \{x\}^c $.  Both implies that $ x\in \{x\}^c $, a contradiction. Hence $ D $ is  $s\kappa_i $-closed and the bispace is pairwise semi-$ T_\omega $. 
\end{proof}

\begin{theorem} \label{36} A bispace $ (X,  \kappa_1, \kappa_2) $ is pairwise semi-$ T_\omega $ if and only if for $ i,j=1,2;i\not=j $;

(i) each subset of $ X $ is the intersection of all  $ s\kappa_i $-closed sets and all $ s\kappa_j $-open sets  containing it and 

(ii)  $ \mathcal {G}_i =  \mathcal {G}_i' $  
where $ \mathcal {G}_i $ and $  \mathcal {G}_i' $ are as in Definition \ref{34}.
\end{theorem}

\begin{proof}
Let the bispace $ (X,  \kappa_1, \kappa_2) $ be pairwise semi-$ T_\omega $. (i): Let $ D\subset X $, then by Theorem \ref{35}, $ D=\bigcap\{X-\{x\}: x\in X, x\not\in D $; if $ \{x\} $ is not $ s\kappa_i $-closed, then it is $ s\kappa_j $-open; $ i,j=1,2;i\not=j\} $; so $ D $ is the intersection of all  $ s\kappa_i $-closed sets and all $ s\kappa_j $-open sets containing it.
(ii) It is proved in Theorem \ref{35}. 

Conversely, assume the conditions (i) and (ii) hold and $ x\in X, \{x\} $ is not $ s\kappa_i $-closed, then $ X-\{x\} $ is not $ s\kappa_i $-open; $ i,j=1,2;i\not=j$ and $ X $ is the only $ s\kappa_i $-open set containing $ X-\{x\} $. So by assumption, $ X-\{x\} $ is $ s\kappa_j $-closed which implies that $ \{x\} $ is $ s\kappa_j $-open. Hence the result follows by heorem \ref{35}.   
\end{proof}

\begin{theorem}
Every semi-door bispace $ (X,  \kappa_1, \kappa_2) $ is pairwise semi-$ T_\omega $.
\end{theorem}

The proof is straight forward.

\begin{example}\label{38} Assume $ X=\{p, q\} $ and $ \kappa_1=\kappa_2=\{\emptyset,X, \{p\}\} $, then $ (X,  \kappa_1, \kappa_2) $ is a bitopological space so a bispace. Obviously, the bispace $ (X,  \kappa_1, \kappa_2) $ is pairwise semi-$ T_\omega  $ but it is not pairwise semi-$ T_1 $.
\end{example}

\begin{theorem}\label{39}
Pairwise semi-$ T_\omega $ bispace $ (X,\kappa_1,\kappa_2) $ is pairwise semi-$ T_0 $.
\end{theorem}
\begin{proof}
Suppose the bispace $ (X,  \kappa_1, \kappa_2) $ is pairwise semi-$ T_\omega $ but it is not pairwise semi-$ T_0 $. Then there exist $ l, m\in\ X; l\not =m $ such that $ \overline{s\{l\}_{\kappa_2}}=\overline{s\{m\}_{\kappa_1}} $. We assert that $ \{l\} $ is not $ s\kappa_2 $-closed. If $ \{l\} $ is $ s\kappa_2 $-closed, then $\overline{s\{l\}_{\kappa_2}}=\{l\}\not=\overline{s\{m\}_{\kappa_1}} $ which contradicts the assumption. So by Theorem \ref{23}, $ \{l\}^c $ is  $ (2-1)sg_\kappa^*$-closed. But the bispace is pairwise semi-$ T_\omega $, so $ \{l\}^c $ is $ s\kappa_1 $-closed. Now we show that $ \{l\}^c $ is not  $ s\kappa_1$-closed. If $ \{l\}^c $ is $s\kappa_1$-closed, then $ m\in \{l\}^c \Rightarrow \overline{s\{m\}_{\kappa_1}} \subset \overline{s(X-\{l\})_{\kappa_1}} = X-\{l\}$. Therefore $ \overline{s\{l\}_{\kappa_2}}\not=\overline{s\{m\}_{\kappa_1}} $ which contradicts the assumption again . Hence the result follows.
\end{proof}

But the converse is not true always as revealed from the Example \ref{14}.

\begin{remark}\label{37}
It is seen that in a topological space, semi-$ T_\frac{1}{2} $ axiom can be placed between semi-$ T_0 $ and semi-$ T_1 $ axioms \cite{BL}. But pairwise semi-$ T_\omega $ axiom in a bispace does not have this property as is evident from  Examples \ref{14} and \ref{26}, rather Examples \ref{14} and \ref{38} show  that pairwise semi-$ T_w $ and pairwise semi-$ T_1 $ axioms in a bispace are independent of each other. It can be easily verified that pairwise semi-$ T_1 $ bispace is pairwise semi-symmetric but converse is not true as in \cite{NL}. In a symmetric topological space, $ T_0, T_\frac{1}{2}, T_1 $ axioms are equivalent \cite{NL}. In a pairwise semi-symmetric bispace, though pairwise semi-$ T_1 $ and pairwise semi-$ T_0 $ axioms are equivalent but pairwise semi-$ T_1 $ bispace may not imply pairwise semi-$ T_w $ as can be seen from the Example \ref{42}.  A topological space is symmetric if and only if each singleton of $ X $ is $ g $-closed \cite{NL}. There exists bispace which is pairwise semi-symmetric but singletons are not $ (j-i)sg_\kappa^* $-closed; $ i,j=1,2;i\not=j $ can be seen from Example \ref{42}.   
\end{remark}

\begin{definition}\label{40}  A bispace $ (X,  \kappa_1, \kappa_2) $ is said to be pairwise strongly  semi-symmetric if each singleton of $  X  $  is $ (j-i)s g_\kappa^*$-closed, $ i,j=1,2; i\not=j $.
\end{definition}

\begin{theorem}\label{41} A  pairwise strongly  semi-symmetric bispace is pairwise semi-symmetric.
\end{theorem}

\begin{proof}
Suppose a bispace $ (X,  \kappa_1, \kappa_2) $ is pairwise strongly  semi-symmetric and $ l \in \overline{s\{m\}_{\kappa_i}}$, but $ m\not \in\overline{s\{l\}_{\kappa_j}} $ for $ l, m \in X, i,j=1,2; i\not=j$. So there is a $ s\kappa_j $-closed set $ F \supset \{l\} $ such that $ m \not \in F \Rightarrow \{m\} \subset F^c,$ a $ s\kappa_j $-open set. For pairwise  strongly semi-symmetric bispace, $ \{m\} $ is $ (j-i)sg_\kappa^* $-closed, then there is a $ s\kappa_i $-closed set $ F' \supset \{m\} $ such that $ F' \subset F^c $. So $ \overline{s\{m\}_{\kappa_i}} \subset F' $ and $ l \in \overline{s\{m\}_{\kappa_i}} \subset F' \subset F^c $, a contradiction. Hence the result follows.
\end{proof}

But the converse may not be true as shown by the following Example.

\begin{example}\label{42} Suppose $ X = R - Q $ and $ \kappa_1=\kappa_2 = \{X, \emptyset, G_i \}$ where $ \{G_i\}$ are the countable subsets of $ X $. So $ (X,  \kappa_1, \kappa_2) $  is a bispace but not a bitopological space. Suppose $ l, m \in X $.   If $ l \not = m $ then $ l\in \overline{s\{m\}_{\kappa_i}} $ implies that $ l $ is a $ s\kappa_i $-limit point of $ \{m\} $. But $ \{l\} $ is a $ s\kappa_i $-open set containing $ l $ which does not intersect $ \{m\} $. Hence $ l $ cannot be a  $ s\kappa_i $-limit point of $ \{m\} $. So we must have $ l = m $ and in that case $ m\in \overline{s\{l\}_{\kappa_j}} $. Hence the bispace is pairwise semi-symmetric. But for each singleton $ \{l\}$ in $ X ,   sker_j(l) = \{l\} $ and $ \{l\} $ is not $ s\kappa_i $-closed. So no singleton is $ (j-i)sg_\kappa^* $-closed. Hence the result follows.
\end{example}

\begin{corollary}\label{43}
A pairwise strongly  semi-symmetric pairwise semi-$ T_\omega $ bispace is pairwise semi-$ T_1 $. 
\end{corollary}

Proof follows from Theorems \ref{32}, \ref{39}, \ref{41}. 

\begin{remark}\label{44}
If the bispace $ (X,  \kappa_1, \kappa_2) $ is pairwise strongly  semi-symmetric, we have the relation among the axioms ``Pairwise semi-$ T_\omega $ axiom $\Rightarrow $ pairwise semi-$ T_0 $ axiom $\Rightarrow $ pairwise semi-$ T_1 $ axiom." 
\end{remark}


\begin{thebibliography}{99}\baselineskip=16pt
%%%%%%%%%%%%%%%%%%%%%%%%%%%%%%%%%%

\bibitem{AD}  A. D. Alexandroff, \textit{Additive set functions in abstract space, } Mat. Sb. (N.S.) \textbf{8}(50) (1940), 307-348 (English, Russian Summary).


\bibitem{AS}  A. K. Banerjee and P. K. Saha, \textit{Bispace group}, \textit{Int. J. Math. Sci and Engg. Applies.}, \textbf{5} (V), (2011), 41-47.

\bibitem{BS}  A. K. Banerjee and P. K. Saha, \textit{Semi open sets in bispace}, \textit{CUBO A Mathematical Journal}, \textbf{17} (1), (2015), 99 - 106. 


\bibitem{BP}  A. K. Banerjee and P. K. Saha, \textit{Quasi-open sets in bispaces}, \textit{Gen. Math. Notes}, \textbf{30} (2), (2015), 1 - 9.


\bibitem{BL} P. Bhattacharyya and B. K. Lahiri. \textit{Semi-generalised closed sets in topology,} \textit{Indian J Math.} \textbf{29} (1987), 376-382. 



\bibitem{CA} A. Cs$\acute{a}$sz$ \acute{a} $r ,  \textit{Generalized open sets in generalized topologies},  \textit{Acta Math. Hungar.,}   \textbf{106} (12) (2005), 53-66.

\bibitem{CG}  Chatterjee, B. C., Ganguly,  Adhikary,  \textit{ A  Text Book of Topology ,}\textit{ Asian Books Private Limited, New Delhi-110 002}.


\bibitem{DR}  P. Das and  M.A. Rashid, \textit{$g^*$-closed sets and a new separation axioms in Alexandroff spaces,} \textit{Archivum Mathematicum (BRNO), Tomus} \textbf{39} (2003), 299-307. 



\bibitem{RD} P. Das and M. A. Rashid,\textit{  Semi-g*-closed sets and a new Separation Axiom in the spaces,} \textit{Bulletin of the Allahabad Mathematical Society.},\textbf{19} (2004), 87-98.

\bibitem{PR} P. Das and M. A. Rashid, \textit {Semi-closed sets in a space,}  \textit{Bull. Gauhati Univ. Math. Ass.} \textbf{6} (1999),  27-35. 



\bibitem{WD}  W. Dunham,  \textit{$ T_\frac{1}{2}$-$spaces $,}\textit{ Kyungpook  Math. J.,} \textbf{17}(2) (1977), 161-169. 

\bibitem{DM} J. Dontchev and H. Maki,  \textit{On $ sg$-closed  sets and semi-$\lambda$-closed sets  }, \textit{Questions Ans. Gen. Topology}, \textbf{15(2)} (1997), 259-266.

\bibitem{JK}  J. C. Kelly,  \textit{Bitopological Spaces,}\textit{ Proc. London Math, Soc.(3),} \textbf{13} (1963), 71-89.

\bibitem{LD}   B. K. Lahiri and   P. Das, \textit{Semi-open  set in a space}, \textit {Sains Malaysiana} \textbf {24}(4)  1995: 1-11.

\bibitem{LK}   B. K. Lahiri and   P. Das, \textit{Certain bitopological concepts in a bispace}, \textit {Soochow Journal of Mathematics}, \textbf {27} (2) (2001), 175-185.

\bibitem{NL}  N. Levine,  \textit{Generalised closed sets in topology}, Rend. Cire.  Mat. Palermo \textbf{19}(2) (1970), 89-96.

\bibitem{MBD} H. Maki, K. Balachandran and R. Devi,  \textit{Remarks on semi-generalised closed  sets and generalised semi closed sets},  \textit{Kyungpook Math. J.}, \textbf{36(1)} (1996), 155-163.

\bibitem{RL} I. L. Reilly,  \textit{On bitopological separation properties}, Nanta  Mathematica,  \textbf{5} (1972), 14-25.


\bibitem{MS} M. S. Sarsak,  \textit{New separation axioms in generalized topological spaces}, \textit{Acta Math. Hungar.,} \textbf{132(3)} (2011), 244-252.

\bibitem{TA} O. A. EI-Tantawy, H. M. Abu-Donia,  \textit{Generalised separation axioms in  bitopological spaces}, \textit{The Arabian Journal for Science and Engineering,}  \textbf{30}  (1A) (2005), 117-129.
 

  
%%%%%%%%%%%%%%%%%%%%%%%%%%%%%%%%
\end{thebibliography}
\end{document}